\documentclass[a4paper,11pt]{amsart}
\usepackage[utf8]{inputenc}
\usepackage{graphicx}
\usepackage[english]{babel}
\usepackage{setspace}
\usepackage{amsmath,amssymb}
\usepackage{amsfonts}
\usepackage{booktabs}
\usepackage{amsthm}
\usepackage{gnuplot-lua-tikz}
\usepackage{mathtools}
\usepackage[colorlinks=true]{hyperref}
\usepackage{tikz}
\usetikzlibrary{graphs}
\usetikzlibrary{arrows.meta} 

\newtheorem{theorem}{Theorem}

\newtheorem{definition}{Definition}

\theoremstyle{remark}
\newtheorem{remark}{Remark}
\newcommand{\ee}[1]{\mathrm{e}^{#1}}
\newcommand{\nth}{n\textsuperscript{th}}
\newcommand{\partialO}[2]{\frac{\partial #1}{\partial #2}}
\newcommand{\partialT}[2]{\frac{\partial^2 #1}{\partial #2 ^2}}

\title[$\nth$ survival function and $\nth$ first passage time distribution]{A recursive formula for the $\nth$ survival function and the $\nth$ first passage time distribution for jump and diffusion processes. Applications to the pricing of $\nth$-to-default CDS.}
\author{Alessio Lapolla}
\date{}
\email{alessiolapolla@gmail.com}
\begin{document}
\maketitle
\begin{abstract}
  We derive some rather general, but complicated, formulae to compute the survival function and the first passage time distribution of the $\nth$ coordinate of a many-body stochastic process in the presence of a killing barrier. First we will study the case of two coordinates and then we will generalize the results to three or more coordinates. Even if the results are difficult to implement, we will provide examples of their use applying them to a physical system, the single file diffusion, and to the financial problem of pricing a $\nth$-to-default credit default swap ($\nth$-CDS).
\end{abstract}
\section{Introduction}
The first passage time problem refers to the characterization of statistics concerning the first time a stochastic process interacts with a boundary. This intuitive definition allowed its application to the most diverse fields like: finance~\cite{brigo_interest_2006,glasserman_monte_2010}, neuroscience~\cite{bachar_stochastic_2013}, physics~\cite{redner_guide_2007, bray_persistence_2013}, psychology~\cite{navarro_fast_2009}, biology~\cite{lawley_distribution_2020}, etc\ldots In the case of a single process (in many dimensions as well) several studies have been and are being performed investigating a great varieties of models, on the other hand studies on systems with two or more coordinates are less common. One of the first fully general results involves the two-dimensional correlated Wiener process~\cite{domine_first_1993, kou_first-passage_2016, sacerdote_first_2016}, even for this model (arguably the most important and fundamental) the solution is rather complicated. Monte Carlo simulations are a powerful and adaptable tool for the investigation of first passage time problems. However, they sometimes suffer of discretization errors~\cite{kloeden_numerical_1992}, therefore the investigation into alternative calculation methods has some practical applications as well.\\
Given a probability space $(\Omega,\mathcal{F},\mathbb{P})$ with right-continuous complete filtration $\mathcal{F}_t$ that supports a stochastic process $\mathbf{X}_t$, describing $N$ different coordinates, on the measurable state space $(R,\Sigma)$; a first passage time problem requires the existence of several hitting times  $\tau_i=\inf\{t>0 | E(\mathbf{X}_{it},\partial R)\}$ with respect to the filtration, one for each coordinate forming $\mathbf{X}_t$. Where the writing $E(\mathbf{X}_{it},\partial R)$ represent a general event involving both the process and the boundary of the measurable space $\partial R$.\\
The boundary can be classified in three different categories~\cite{sacerdote_first_2016}:
\begin{enumerate}
\item killing: when the coordinate that hits the boundary disappears and does not contribute anymore to the system's dynamics;
\item absorbing: when the coordinate hitting the boundary sticks to it but continues interacting with the others;
\item crossing: when the system's dynamics is not affected by the hitting.
\end{enumerate}
In this article we will only treat the case of a killing boundary and we will derive an exact formula for the survival function and the first passage time distribution of the $\nth$ coordinate being killed. Or, in other words, we will derive an order statistics for the first passage time distributions of a many-body system~\cite{majumdar_extreme_2020}. This setup is peculiar; at each hitting time the ``equations of motion'' of the system change since the killed coordinate does not interact anymore with the others and a dimensionality reduction restricts the state space.\\ 
We will first analyze the simpler case with only two coordinates and we will apply our results to both a jump and a diffusion process, the latter process, the single file diffusion, is a classic problem in statistical physics. Then we will generalize to many coordinates and will use our results to price a $\nth$-to-default credit default swap ($\nth$-CDS).

\section{The two coordinates case}
The simplest case is the two coordinate system. Let $(\Omega,\mathcal{F}, \mathcal{F}_t, \mathbb{P})$ be a probability space on which we define the two coordinates continuous time c\`adl\`ag Markovian stochastic process $\mathbf{X}_t=(X_{1t},X_{2t})$ taking values in $R^2=R\times R$ with boundary $\partial R$. For example we can assume $R$ to be a subset of $\mathbb{R}$, like $\mathbb{N}$ or an interval, and $\partial R$ to be a natural number or the extremes of the interval. $R$ could also be a multi-dimensional space like a ball in $\mathbb{R}^n$, the boundary $\partial R$ could be its surface (the hyper-sphere), and thus both $X_{1t}$ and $X_{2t}$ would be multi-dimensional processes in this case. Then let
  \begin{align}
    &\tau_1=\inf\{t>0| X_{1t} \in \partial R\}, &\tau_2=\inf\{t>0|X_{2t} \in \partial R\},
  \end{align}
  if we assume that a coordinate is killed upon reaching the boundary, we can define two additional hitting times:
  \begin{align}
    &\tau_m=\inf\{t>0 |\tau_1\wedge \tau_2\}, \\
    &\tau_M=\inf\{t>0 |\tau_1\vee \tau_2\}= \\
    &\inf\{t>0| (X_{1t} \in \partial R\ \cap \tau_1\geq\tau_2) \cup (X_{2t}\in \partial R \cap \tau_2>\tau_1)\}\nonumber
  \end{align}
  where $\tau_m$ represents the first time one of the coordinates hits the boundary while $\tau_M$ the time the last coordinate is finally killed. Note that, in general, we assume that the two coordinates somewhat interact via a potential, via a copula structure, or via additional random processes. Therefore, the system is just described by a single coordinate $X_{rt}$  whose evolution is independent from the other coordinate for $t>\tau_m$. We always assume that all hitting times exist, are finite, and are inaccessible stopping times under the filtration $\mathcal{F}_t$.\\
  We define with $P^2(\mathbf{X},t|\mathbf{X}_0,t_0)$ the transitional probability density or mass (depending on the system we are considering), that is the probability of finding the system in $\mathbf{X}$ at time $t$ conditioned on it being in $\mathbf{Y}$ at time $t_0$ (for the rest of the article $t_0=0$). Then the survival function of both $X$s $S^2(t)=\mathbb{P}(t\leq\tau_m)$ is
  \begin{equation*}
    S^2(t)=\int_{R^2}d\mathbf{X}P^2(\mathbf{X},t|\mathbf{X}_0,t_0)\quad \text{or}\quad S^2(t)=\sum_{\mathbf{X}\in R}P^2(\mathbf{X},t|\mathbf{X}_0,t_0)
  \end{equation*}
  and the first passage time distribution can be computed via the time derivative $F^2(t)=-\frac{dS^2(t)}{dt}$~\cite{redner_guide_2007}, assuming that the derivative exists. These functions completely describe the system up to time $\tau_m$, the first results we will soon present provide formulae for describing the system up to $\tau_M$. Before proceeding, we define the following functions:
  \begin{itemize}
  \item the single coordinate transitional probability density or mass function for the $i=1,2$ coordinate: $P^1_i(x,t|y,s),\,s<t$;
  \item the marginal transitional probability density or mass of a coordinate for $s<t<\tau_m$: $P^2_i(x,t|y,s)$;
  \item the conditional probability density or mass of the $i$-coordinate conditioned on the killing of the other at time $t$: $P^2_i(x,t|X_j \in \partial R,t)$;
  \item the marginal survival probability $S_i(t)$ and the marginal first passage time density (if it exists) $F_i(t)$ for $t<\tau_m$.
  \end{itemize}
  All these functions are assumed to exist and to be computable analytically or at least numerically. The first result is the following: 
  
  \begin{theorem}
    \label{theo:two}
    Assuming that all $P$s are probability density function and that $S_1(t)$ and $S_2(t)$ are absolutely continuous, then the last particle survival function $S^1(t)$ can be computed using the following formula:
    \begin{equation}
      \label{eq:S1_general}
      \begin{aligned}
        &S^1(t)=S^2(t)+\int_Rdx\int_0^tds \int_Rdy P^1_1(x,t|y,s)P^2_1(y,s|X_2\in \partial R,s)F_{2}(s)+\\
        &\int_Rdx\int_0^tds \int_Rdy P^1_2(x,t|y,s)P^2_2(y,s|X_1\in \partial R,s)F_{1}(s),
      \end{aligned}
    \end{equation}
    and $F^1(t)=-\frac{dS^1(t)}{dt}$.
  \end{theorem}
  \begin{proof}
    By definition the law of either coordinate $1$ or $2$ is described by $P^2_i(x,t|x_0,s)$ for $s<t\leq\tau_m$ and $i=1,2$; while the law of the last coordinate remaining by $P^1_1(x,t|x_0,s)$ if $\tau_2<\tau_1$ or $P^1_2(x,t|x_0,s)$ if $\tau_2>\tau_1$ for $t>s>\tau_m$. Therefore the probability of survival of the last particle is composed by two mutually exclusive and complementary events:
    \begin{equation}
      S^1(t)=\mathbb{P}(t<\tau_M)=\mathbb{P}(t\leq\tau_m)+\mathbb{P}(\tau_m<t<\tau_M)=S^2(t)+\mathbb{P}(\tau_m<t<\tau_M),
    \end{equation}
    where the first addend encompasses also the case $\mathbb{P}(\tau_1=\tau_2)$, even when this probability is larger than zero. The second addend is in turn formed by two mutually exclusive events, and thanks to the law of total probability:
    \begin{equation}
      \label{eq:total_prob}
      \begin{aligned}
        \mathbb{P}(\tau_m<t<\tau_M)=&\mathbb{P}(\tau_m<t<\tau_M|\tau_2=\tau_m)\mathbb{P}(\tau_2=\tau_m)+\\
        &\mathbb{P}(\tau_m<t<\tau_M|\tau_1=\tau_m)\mathbb{P}(\tau_1=\tau_m).
      \end{aligned}
    \end{equation}
    Now, since the first killing happened at a certain time $s<\tau_m<t$ we have that
    \begin{equation}
      \mathbb{P}(\tau_2=\tau_m)=\mathbb{P}(t>\tau_2)=1-\mathbb{P}(t\leq\tau_2)=1-S_2(t)
    \end{equation}
    is the default probability of $X_2$, and
    \begin{equation}
      \begin{aligned}
        &\mathbb{P}(\tau_m<t<\tau_M|\tau_2=\tau_m)=\frac{\int_0^tds \mathbb{P}(s<t<\tau_M|\tau_2=s)F_{2}(s)}{1-\mathbb{P}(t\leq\tau_2)}=\\
        &\frac{\int_0^tds \int_Rdy \mathbb{P}(t<\tau_M|y,s)P^2_1(y,s|X_2 \in \partial R,s)F_{2}(s)}{1-\mathbb{P}(t\leq\tau_2)}=\\
        &\frac{\int_Rdx\int_0^tds \int_Rdy P^1_1(x,t|y,s)P^2_1(y,s|X_1 \in \partial R,s)F_{2}(s)}{1-\mathbb{P}(t\leq\tau_2)};
      \end{aligned}
    \end{equation}
    where we computed the conditional average over the killing time of $X_2$ in the first equality, used the Markov property in the second, and used the definition of survival function in the third. The same applies\emph{ mutatis mutandis} to the second added of Eq.~\eqref{eq:total_prob}. Thus, putting all together, we finally obtain Eq.~\eqref{eq:S1_general}.
  \end{proof}
  \begin{remark}
    In the case in which we are dealing with probability masses the two integrals over $R$ become sums, while if the marginal survival functions fail to be absolutely continuous then the time integral needs to be rewritten as a Lebesgue–Stieltjes integral in Eq.~\eqref{eq:S1_general}.
  \end{remark}
Eq.~\eqref{eq:S1_general} has an intuitive explanation. The time the second coordinate is killed is given by the time the first coordinate is killed (the first addend) plus the time the second coordinate needs to reach the boundary. This last time is given by the average over the time the first coordinate is killed and the average over the position of the remaining particle at the first killing time (the symmetric second and third addends).

  \subsection{A jump process example}
  The fundamental jump process is the homogeneous Poisson process with intensity $\lambda$ whose distribution is defined as $\operatorname{Pois}(\lambda t)$. Its probability mass is defined as:
  \begin{equation}
    \label{eq:poission1D}
    P_\lambda(k,t)=\frac{(\lambda t)^k\ee{-\lambda t}}{k!},
  \end{equation}
  its survival function with killing boundary at $M>0$ is
  \begin{equation}
    \label{eq:survival1D}
    S_\lambda(t)=\sum_{k=0}^{M-1} P_\lambda(k,t),
  \end{equation}
  and by derivation the first passage time density is
  \begin{equation}
    \label{eq:fpt_1D}
    F_\lambda(t)=\sum_{k=0}^{M-1}\frac{\lambda^k t^{k-1}}{k!}(\lambda t-k)\ee{-\lambda t}.
  \end{equation}
  There are several generalization to many dimensions, we will consider one of the easiest ones, characterized by a simple parallelism with the two dimensional correlated Wiener process. The bivariate Poisson process~\cite{holgate_estimation_1964, kotz_discrete_2005} is formed by two correlated Poisson processes, $X_1$ and $X_2$, defined using three independent Poisson processes with intensities $\lambda_1$, $\lambda_2$, and $\lambda_{12}$ as
  \begin{align}
    &X_1=Y_1+Y_{12} &X_2=Y_2+Y_{12}.
  \end{align}
  This model has found some applications in the analysis of sports data~\cite{karlis_analysis_2003}.
  The joint probability mass function of this two dimensional process reads:
  \begin{equation}
    \label{eq:poisson_join_2D}
    P(x_1,x_2,t)=\frac{(\lambda_1t)^{x_1}(\lambda_2t)^{x_2}}{x_1! x_2!}\ee{-(\lambda_1+\lambda_2+\lambda_{12})t}\sum_{k=0}^{\min(x_1,x_2)}\binom{x_1}{k}\binom{x_2}{k}k!\left(\frac{\lambda_{12}}{\lambda_1\lambda_2}\right)^k\frac{1}{t^k},
  \end{equation}
  its marginal distributions are:
  \begin{align}
    &X_1\sim\operatorname{Pois}((\lambda_1+\lambda_{12})t), &X_2\sim\operatorname{Pois}((\lambda_2+\lambda_{12})t),
  \end{align}
  the Pearson correlation between $X_1$ and $X_2$ is given by:
  \begin{equation}
    \rho=\frac{\lambda_{12}}{\sqrt{(\lambda_1+\lambda_{12})(\lambda_2+\lambda_{12})}},
  \end{equation}
  therefore $\lambda_{12}$ plays a role similar (but not identical) to the correlation coefficient in the two dimensional correlated Wiener process.\\
  The conditional distribution is:
  \begin{equation}
    P^2_1(X_1=x,t|X_2=y,t)=\ee{-\lambda_1t}\sum_{j=0}^{\min(x,y)}\binom{y}{j}\left(\frac{\lambda_{12}}{\lambda_2+\lambda_{12}}\right)^j\left(\frac{\lambda_{2}}{\lambda_2+\lambda_{12}}\right)^{y-j}\frac{(\lambda_1t)^{x-j}}{(x-j)!},
  \end{equation}
  that is convolution between a Poisson and a binomial distribution~\cite{kotz_discrete_2005}.\\
  If we assume a killing barrier at $M>0$, the survival function describing the probability that both coordinates are below this barriers is
  \begin{equation}
    \label{eq:survivalp2d_first}
    S^2(t)=\sum_{x=0}^{M-1}\sum_{y=0}^{M-1}P(x,y,t)
  \end{equation}
  and the first passage time distribution of the two dimensional system (describing the time of first exit) is
  \begin{equation}
    \label{eq:fpt_first}
    \begin{aligned}
      &F^2(t)=\sum_{x=0}^{M-1}\sum_{y=0}^{M-1}\frac{\lambda_1^x\lambda_2^y}{x! y!}\sum_{k=0}^{\min(x,y)}\binom{x}{k}\binom{y}{k}\left(\frac{\lambda_{12}}{\lambda_1\lambda_2}\right)^kk!\\
      &[(\lambda_{1} + \lambda_{12} + \lambda_{2}) t + k-x - y] t^{x+y-k-1} \ee{-(\lambda_{1}+\lambda_{12}+\lambda_{2})t}.
    \end{aligned}
  \end{equation}
  Now if we assume, for example, that $X_1=M$ first at $\tau_m$; then for subsequent times the Poisson process $X_2$ will follow the law given by $\operatorname{Pois}(\lambda_2 t)$ and not by $\operatorname{Pois}((\lambda_2+\lambda_{12})t)$ anymore, since, with the killing of $X_1$, the meaning of $Y_{12}$ as the coupling parameters ceases to make sense. The bivariate Poisson model allows for the simultaneous killing of both $X$s as well.\\
  Thus our formula~\eqref{eq:S1_general} dictates that the survival function of the last particle to exit is:
  \begin{equation}
    \label{eq:survivalp2d_last}
    \begin{aligned}
      &S^1(t)=S^2(t)+\sum_{j=0}^{M-1} \int_0^td\tau \sum_{i=0}^{j} P_{\lambda_1}(j,t|i,\tau)P^2_1(i,\tau|M,\tau)F_{\lambda_2+\lambda_{12}}(\tau)\\
      &+\sum_{j=0}^{M-1} \int_0^td\tau \sum_{i=0}^{j} P_{\lambda_2}(j,t|i,\tau)P^2_2(i,\tau|M,\tau)F_{\lambda_1+\lambda_{12}}(\tau)      
    \end{aligned}
  \end{equation}
  where $i<M$.
  Now expanding and collecting the terms involving time
  \begin{equation}
    \begin{aligned}
      &\sum_{j=0}^{M-1}\sum_{i=0}^{j}\frac{\lambda_1^{j-i}}{(j-i)!}\sum_{k=0}^i\binom{M}{k}\left(\frac{\lambda_{12}}{\lambda_{12}+\lambda_2}\right)^k\left(\frac{\lambda_{2}}{\lambda_{12}+\lambda_2}\right)^{M-k} \frac{\lambda_1^{i-k}}{(i-k)!}\sum_{l=0}^{M-1}\frac{(\lambda_2+\lambda_{12})^l}{l!}\\
      &\int_0^t d\tau \ee{-\lambda_1(t-\tau)}\ee{-\lambda_1\tau}(t-\tau)^{j-i}\tau^{i-k+l-1}[(\lambda_2+\lambda_{12})\tau-l]\ee{-(\lambda_2+\lambda_{12})\tau}.
    \end{aligned}
  \end{equation}
  Therefore, we can define two helping integrals
  \begin{equation}
    H_1(t)=(\lambda_{12} + \lambda_{2})\ee{-\lambda_1t}\int_0^t d\tau (t-\tau)^{j-i}\tau^{i-k+l}\ee{-(\lambda_{12}+\lambda_{2})\tau}
  \end{equation}
  and
  \begin{equation}
    H_2(t)=-l\ee{-\lambda_1t}\int_0^t d\tau (t-\tau)^{j-i}\tau^{i-k+l-1}\ee{-(\lambda_{12}+\lambda_{2})\tau}
  \end{equation}
  whose solutions~\cite{gradshteyn_i_s_and_ryzhik_i_m_table_2007} can be written down analytically:
  \begin{equation}
    \int_0^t d\tau (t-\tau)^{\alpha}\tau^{\beta}\ee{-\lambda\tau}=B(\alpha+1,\beta+1)t^{\alpha+\beta+1}\, _1F_1(\beta+1,\alpha+\beta+2;-\lambda t)
  \end{equation}
  with $\alpha,\beta\geq0$, and the caveat that $H_2(t)=0$ if $l=0$.  $B(\alpha,\beta)$ denotes the Beta function and $_1F_1(p,q,t)$ the confluent hypergeometric function of the first kind. The first passage time  distribution can be obtained by derivation. See Fig.~\ref{fig:survivalp2D} for a comparison between the formulae obtained and a Monte Carlo simulation.
  \begin{figure}
   \centering
   \includegraphics[width=0.95\textwidth]{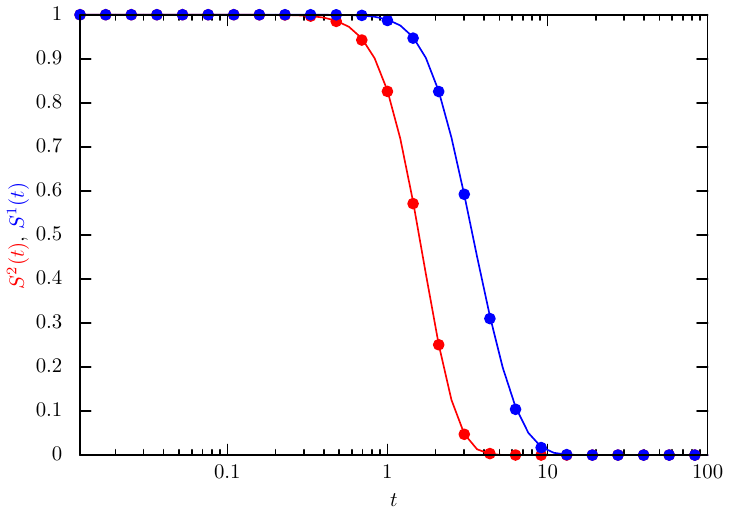}
   \caption{Comparison between a simulation (dots obtained averaging over $10,000$ realization) and the theory (solid lines) for a bivariate Poisson model with $\lambda_1=1$, $\lambda_2=2$, $\lambda_{12}=0.8$, and $M=5$. The red lines show the result for $S^2(t)$ (Eq.~\eqref{eq:survivalp2d_first}) while the blue ones for $S^1(t)$(Eq.~\eqref{eq:survivalp2d_last}).}\label{fig:survivalp2D}
 \end{figure}
  
  \subsection{A diffusion process example}
  The diffusion problem we tackle is the most simple variation of a single file diffusing in a box. This is a classic problem~\cite{harris_diffusion_1965,jepsen_dynamics_1965} that found numerous applications to various fields~\cite{hummer_water_2001, li_effects_2009}, in particular the exit-time of colloids in a single file has been analyzed both theoretically in~\cite{ryabov_single-file_2013} and experimentally in~\cite{locatelli_single-file_2016}.  In our variation, two identical Brownian particles diffuse in a box with a reflecting boundary condition at $0$ and a killing boundary condition at $1$ in the absence of an external potential. The two Brownian particles are subject to a ``hard-core'' interaction that does not allow them to cross each other. The transitional probability density $P^2(x_1,x_2,t|x_{01},x_{02},0)$ satisfies the Fokker-Planck equation (with unitary diffusion coefficient):
\begin{equation}
  \label{eq:single_fileFP}
  \begin{aligned}
    &\partialO{P^2}{t}=\partialT{P^2}{x_1}+\partialT{P^2}{x_2},\\
    &P^2(x_1,x_2,0|x_{01},x_{02},0)=\delta(x_1-x_{01})\delta(x_2-x_{02}),\\
    &\partial_{x_1}P^2(x_1,x_2,0|x_{01},x_{02},0)\left.\right|_{x_1=0}=0,\\
    &P^2(x_1,x_2,0|x_{01},x_{02},0)\left.\right|_{x_1=1}=0,\\
    &P^2(x_1,x_2,0|x_{01},x_{02},0)\left.\right|_{x_2=1}=0,\\
    &(\partial_{x_2}-\partial_{x_1})P^2(x_1,x_2,0|x_{01},x_{02},0)\left.\right|_{x_1=x_2}=0.
  \end{aligned}
\end{equation}
The Fokker-Planck equation can be solved considering the equivalent single particle solution and the ordering operator (represented by the Heaviside theta)
\begin{equation}
  \label{eq:solution_sf_2D}
  P^2(x_1,x_2,t|x_{01},x_{02},0)=2\Theta(x_2-x_1)\sum_{k_1,k_2=0}^\infty\phi_{k_1}(x_1)\phi_{k_2}(x_2)\phi_{k_1}(x_{01})\phi_{k_2}(x_{02})\ee{-\Lambda_{k_1,k_2}t};
\end{equation}
and
\begin{align}
  &\phi_{k}(x)=2\sqrt{2}\cos\left(\frac{(2k+1)\pi x}{2}\right),\\
  &\lambda_k=\frac{(2k+1)^2\pi^2}{4},\\
  &\Lambda_{k_1,k_2}=\lambda_{k_1}+\lambda_{k_2},
\end{align}
where $\phi_k(x)$ and $\lambda_k$ are the eigenfunctions and the eigenvalues of the single particle problem
\begin{equation}
  \label{eq:single_sol}
  P^1(x,t|x_0,0)=\sum_{k=0}^\infty \phi_k(x)\phi_k(x_0)\ee{-\lambda_k t}.
\end{equation}
To simplify the notation we consider a (arguably) more natural initial condition, where the two initial positions are drawn from the uniform distribution $\mathcal{U}$. In this case the transitional probability density can be expressed integrating over the initial conditions Eq.~\eqref{eq:solution_sf_2D} and reads~\cite{lapolla_first_2022}:
\begin{equation}
  \label{eq:single_equilibrium_2D}
   P^2(x_1,x_2,t|\mathcal{U},0)=2\Theta(x_2-x_1)\sum_{k_1,k_2=0}^\infty\phi_{k_1}(x_1)\phi_{k_2}(x_2)\Psi_{k_1,k_2}\ee{-\Lambda_{k_1,k_2}t},
 \end{equation}
 with
 \begin{equation*}
  \Psi_{k_1,k_2}=\frac{8 \, \left(-1\right)^{k_{1}} \left(-1\right)^{k_{2}}}{\pi^{2} {\left(2 \, k_{1} + 1\right)} {\left(2 \, k_{2} + 1\right)}}.
 \end{equation*}
 The transitional probability density of the last particle to exit, the leftmost, is then given by the sum of two terms: one describing the system before the rightmost particle is absorbed $N=2$ and the second describing the system when only the leftmost survives:
 \begin{equation}
   \label{eq:greensf_decomposition}
   P^2_1(x_1,t|\mathcal{U},0)=P^2_1(x_1,t|\mathcal{U},0; N=2)+P^1(x_1,t|\mathcal{U},0; N=1),
 \end{equation}
 the first addend is~\cite{lapolla_bethesf_2020}:
 \begin{equation}
   \label{eq:greensf_N2}
   P^2_1(x_1,t|\mathcal{U},0; N=2)=\sum_{k_1,k_2=0}^\infty V_{k_1,k_2}(x_1)\Psi_{k_1,k_2}\ee{-\Lambda_{k_1,k_2}t},
 \end{equation}
 with~\cite{lapolla_bethesf_2020}
 \begin{equation*}
   V_{k_1,k_2}(x)=2\phi_{k_1}(x)\int_{x_1}^1 dy \phi_{k_2}(y).
 \end{equation*}
 While the second addend is more complicated:
 \begin{equation}
   \label{eq:greensf_N1}
   P^2_1(x_1,t|\mathcal{U},0; N=1)=\int_0^t ds\int_0^1 dy P^1(x,t|y,s)P^2_1(y,s|\mathcal{U},x_2=1,s)F^2(s).
 \end{equation}
 The term $P_1^2(x_1,t|\mathcal{U},0;x_2=1)$ represent the transitional probability density of the first particle at the killing moment $\tau_2$ of the second particle, it reads:
 \begin{equation}
   \label{eq:conditioned_greensf}
   P_1^2(x_1,t|\mathcal{U},0;x_2=1)=\lim_{x_2\to1}\frac{P^2(x_1,x_2,t|\mathcal{U},0)}{P^2_2(x_2,t|\mathcal{U},0;N=2)},
 \end{equation}
 where the denominator is almost identical to Eq.~\eqref{eq:greensf_N2}, with the only difference that now~\cite{lapolla_bethesf_2020}
 \begin{equation*}
   V_{k_1,k_2}(x)=2\phi_{k_1}(x)\int_{0}^{x_1} dy \phi_{k_2}(y).
 \end{equation*}
 The limit in Eq.~\eqref{eq:conditioned_greensf} is a $0/0$ indeterminate form and can be solved using the l'H\^opital rule, therefore Eq.~\eqref{eq:conditioned_greensf} reads:
 \begin{equation}
   P_1^2(x_1,t|\mathcal{U},0;x_2=1)=\frac{2\sum_{k_1,k_2=0}^\infty\phi_{k_1}(x_1)\phi^\prime_{k_2}(1)\Psi_{k_1,k_2}\ee{-\Lambda_{k_1,k2}t}}{\sum_{k_1,k_2=0}^\infty V^\prime_{k_1,k_2}(1)\Psi_{k_1,k_2}\ee{-\Lambda_{k_1,k_2}t}},
 \end{equation}
 while~\cite{lapolla_first_2022}
 \begin{equation}
   F^2(\tau|\mathcal{U},0)=\sum_{k_1,k_2=0}^\infty \Lambda_{k_1,k_2}\Psi_{k_1,k_2}^2\ee{-\Lambda_{k_1,k_2}t}
 \end{equation}
 is the first passage time distribution of the rightmost particle. Expanding all the series and considering the integrals explicitly:
 \begin{equation}
   \label{eq:greensf_N1_2}
   P^2_1(x_1,t|\mathcal{U},0; N=1)=2\sum_{i=0}^\infty\phi_i(x_1)\sum_{k_2=0}^\infty\phi^\prime_{k_2}(1)\Psi_{i,k_2}\sum_{l_1,l_2=0}^\infty \Lambda_{l_1,l_2}\Psi_{l_1,l_2}^2 H(t,i,k_2,l_1,l_2),
 \end{equation}
 where one sum disappears due to the orthonormality of the single particle eigenfunction, and
 \begin{align}
   \label{eq:numerical_integral}
   &H(t,i,k_2,l_1,l_2)=\int_0^t d\tau \frac{\ee{-\lambda_i(t-\tau)}\ee{-(\Lambda_{i,k_2}+\Lambda_{l_1,l_2})\tau}}{C(\tau)},\\
   &C(\tau)=\sum_{j_1,j_2=0}^\infty V^\prime_{j_1,j_2}(1)\Psi_{j_1,j_2}\ee{-\Lambda_{j_1,j_2}\tau}.
 \end{align}
 Eq~\eqref{eq:numerical_integral} can only be evaluated numerically\footnote{The form of the integrand has been chosen for numerical stability, in this form only negative exponentials appear.}. The survival function of the leftmost particle $S^1(t|\mathcal{U})$ can be easily obtained integrating Eq.~\eqref{eq:greensf_decomposition} over $x_1$ (see Fig.~\ref{fig:survivalsf_1}).
 \begin{figure}
   \centering
   \includegraphics[width=0.95\textwidth]{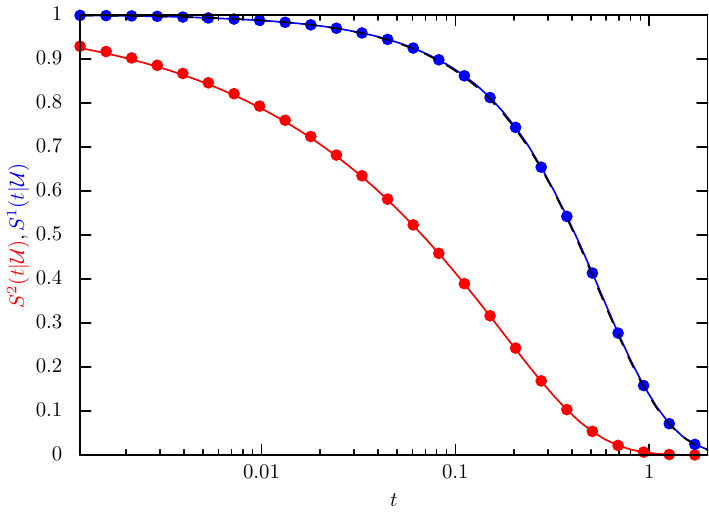}
   \caption{Comparison between a Brownian dynamics simulation ($100,000$ realizations and a time step of $10^{-6}$)  of a single file of two elements and the theoretical results. In particular we show the survival function of the leftmost particle in blue and of the rightmost in red. The dots are obtained using the simulation's results while the solid lines using our formulas. The black dashed line has been obtained using the result in~\cite{locatelli_active_2015}(see Eq.~\eqref{eq:locatelliformula}).}\label{fig:survivalsf_1}
 \end{figure}
 We note that our result in Eq.~\eqref{eq:greensf_N1_2} is identical to the results reported in~\cite{locatelli_active_2015} (dashed black line in Fig.~\ref{fig:survivalsf_1}). Their solution exploits the symmetry of the problem via the ``Reflection Principle'' and reads, for the two particles case:
 \begin{align}
   \label{eq:locatelliformula}
   &S^1(t)=2f(s)(1-s(t))+s(t)^2 &s(t)=\frac{8}{\pi^2}\sum_{k=0}^\infty\frac{\ee{-\lambda_k t}}{(2k+1)^2}
 \end{align}
 where $s(t)$ is the single particle survival function.
 The first passage time distribution $F^1(t|\mathcal{U})=-\partialO{S^1}{t}$ can be obtained using the Leibniz integral rule on Eq.~\eqref{eq:numerical_integral}:
 \begin{equation}
   \label{eq:FPTsf_1}
   F^1(t|\mathcal{U})=F^2(t|\mathcal{U}) + \sum_{i=0}^\infty\chi_i\sum_{k_2=0}^\infty\phi^\prime_{k_2}(1)\Psi_{i,k_2}\sum_{l_1,l_2=0}^\infty \Lambda_{l_1,l_2}\Psi_{l_1,l_2}^2 H^\prime(t,i,k_2,l_1,l_2),
 \end{equation}
 where 
 \begin{align}
   \label{eq:numerical_integral1}
   &H^\prime(t,i,k_2,l_1,l_2)=\lambda_i\int_0^t d\tau \frac{\ee{-\lambda_i(t-\tau)}\ee{-(\Lambda_{i,k_2}+\Lambda_{l_1,l_2})\tau}}{C(\tau)}-\frac{\ee{-(\Lambda_{i,k_2}+\Lambda_{l_1,l_2})t}}{C(t)},\\
   &\chi_i=-\frac{\sqrt{2}\pi(1+ 2i)(-1)^i}{2}
 \end{align}

 \section{The many coordinates case}
 Eq.~\eqref{eq:S1_general} is combinatorial in nature, the final result has been obtained considering all \emph{possible}\footnote{In the single file example the leftmost particle cannot be the first absorbed, in fact Eq.~\eqref{eq:greensf_N1_2} has one less addend than Eq.~\eqref{eq:S1_general}.} ways the last coordinate can be killed.\\
 \begin{figure}
   \label{fig:graph}
   \centering
   \begin{tikzpicture}[scale=1.5]
     \node (111) at (0,0) {AAA};
     \node (011) at (2,2) {DAA};
     \node (101) at (2,0) {ADA};
     \node (110) at (2,-2) {DDA};
     \node (001) at (4,3) {DDA};
     \node (010) at (4,0.5) {DAD};
     \node (100) at (4,-3) {ADD};
     \node (000) at (6,0) {DDD};
     \graph[edges={thick}]
     { (111) ->{ (011)->{(001)->(000),(010)->(000)}, (101)->{(100)->(000),(001)},(110)->{(100),(010)}};
       (111) ->[bend left] (001);
       (111) ->(010);
       (111) ->[bend right] (100);
     };
 \end{tikzpicture}
    \caption{The graph $\Gamma$ relative to the trivariate Poisson model in Section~\ref{sec:tri_poisson}. $A$
labels the alive coordinates and $D$ the killed ones.}
\end{figure}
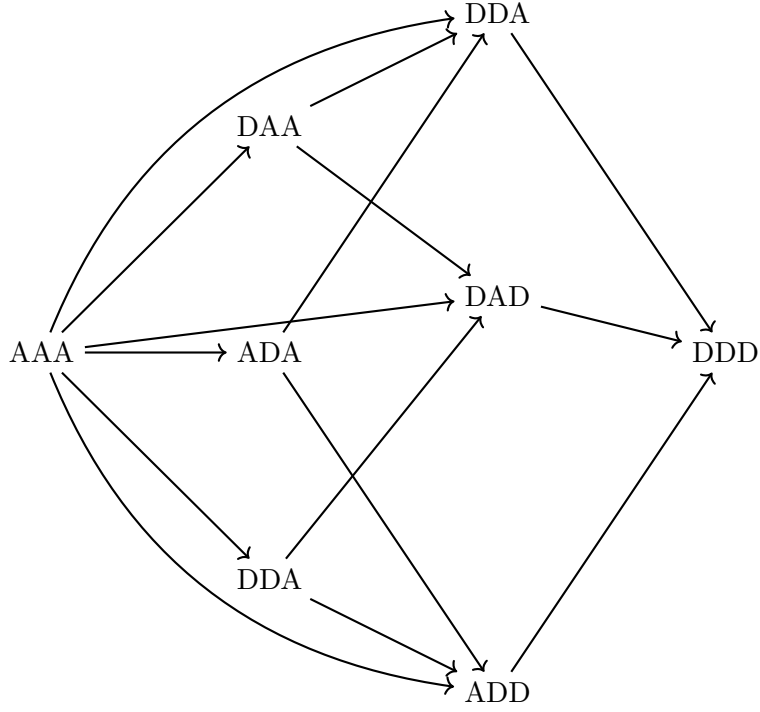
Consider a system of $N$ Markovian coordinates and a directed graph $\Gamma$ (see Fig.~\ref{fig:graph} for an example), each node corresponds to a finite binary string of $N$ elements, where $A$ labels the alive coordinates and $D$ the killed ones. Then defining $\mathcal{G}$ as the set of all paths that connect the different states and $\mathcal{G}_n$ the set of all paths connecting the initial state (the one labeled by a string with only $A$s) to the states with exactly $n$ $A$s. Then we can define:
 \begin{definition}[Path survival function contribution]
   Let $s\in\Gamma$ be a node with only $n$ surviving coordinates and let $g \in \mathcal{G}_n$ be a path passing through the following time-ordered collection of nodes and hitting times pairs: \\ $\{(s_0,0),(s_1,\tau_1),\ldots,(s_{n-1},\tau_{n_1}),(s,t)\}$, where $s_0$ is the initial state where no coordinate has been killed yet. Then the contribution to the $\nth$ survival function given by $g$ is defined by the following integral:
   \begin{equation}
     \label{eq:partial_nth_survival}
     \begin{aligned}
       I_g(t)=&\int_{R^n}d\mathbf{x}\int_0^t d\tau_{n-1}\int_0^{\tau_{n-1}}d\tau_{n-2}\cdots\int_0^{\tau_2}d\tau_1 \int_{R^n}d\mathbf{y}\\
              &P^n(\mathbf{x},t|\mathbf{y},\tau_{n-1}) \times\\
       &P^{N}_n(\mathbf{y},\tau_{n-1}|X_{(n-1)\textsuperscript{th}}\in\partial R,\tau_{n-1};\ldots;X_{1\textsuperscript{st}}\in\partial R,\tau_1)\times\\
       &F_{n+1}(\tau_{n-1})\cdots F_{N}(\tau_1);
     \end{aligned}     
   \end{equation}
   where the subscripts of the $X$s label the order of killings given by $g$.
 \end{definition}
 \begin{remark}
   In Eq.~\eqref{eq:partial_nth_survival} $F_{N}$ denotes the first passage time distribution describing the first coordinate being killed, $F_{N-1}$ denotes the first passage time distribution describing the second coordinate being killed in the original system, that is the first coordinate to be killed in a system of $N-1$ coordinates since it is constrained on being without $X_{1\textsuperscript{st}}$ for times larger than $\tau_1$ and so on\ldots
 \end{remark}
 Thus we can finally prove the following:
 \begin{theorem}
 \label{theo:many}
   The $\nth$ survival function can be computed with the following formula:
   \begin{equation}
     \label{eq:survival_n_general}
     S^n(t)=S^{n+1}(t) + \sum_{g\in\mathcal{G}_n}I_g(t),
   \end{equation}
   while the $\nth$ first passage time distribution can be obtained via derivation $F^n(t)=-\partialO{S^n}{t}$.
 \end{theorem}
 \begin{proof}
   Given $S^N(t)$, $S^{N-1}(t)$ can be obtained using the same steps in Theorem~\ref{theo:two} considering all the possible ways the first coordinate may be killed. Then $S^{N-1}(t),\ldots,S^n(t)$ can be obtained recursively enumerating all the possible paths that end up in states with only $n$ coordinates not killed yet.
 \end{proof}
 Due to its recursive nature, the result in Eq.~\eqref{eq:survival_n_general} is rather cumbersome since requires the calculation of all the path contributions, not only of the ones ending with $n$ coordinates but also for all the others with more remaining coordinates.
 \subsection{The trivariate Poisson process}
 \label{sec:tri_poisson}
 The straightforward generalization of the previously analyzed bivariate Poisson model is the trivariate Poisson model~\cite{kawamura_structure_1976}:
 \begin{equation}
  \label{eq:trivariate}
  \begin{aligned}
    &X_1=Y_1+Y_{12}+Y_{13},\\
    &X_2=Y_2+Y_{12}+Y_{23},\\
    &X_3=Y_3+Y_{13}+Y_{23},
  \end{aligned}
\end{equation}
all $Y$s are Poisson processes with intensity $\lambda_i$ or $\lambda_{ij}$. It is also possible to extend the model adding another Poisson process $Y_{123}$ to the three equations in Eq.~\eqref{eq:trivariate} to consider a three-way correlation contribution as well. Defining $a=\lambda_1+\lambda_2+\lambda_3+\lambda_{12}+\lambda_{13}+\lambda_{23}$, the probability mass function reads:
\begin{equation}
  \label{eq:trivariate_poisson_PMF}
  \begin{aligned}
    &P^3(X_1=x_1,X_2=x_2,X_3=x_3,t|X_1=0,X_2=0,X_3=0,0)=\ee{-at}\\
    &\sum_{y_{12},y_{13},y_{23}\in D}\frac{(\lambda_1t)^{x_1-y_{12}-y_{13}}(\lambda_2t)^{x_2-y_{12}-y_{23}}(\lambda_3t)^{x_3-y_{13}-y_{23}}(\lambda_{12}t)^{y_{12}}(\lambda_{13}t)^{y_{13}}(\lambda_{23}t)^{y_{23}}}{(x_1-y_{12}-y_{13})!(x_2-y_{12}-y_{23})!(x_3-y_{13}-y_{23})!y_{12}!y_{13}!y_{23}!}
  \end{aligned}
\end{equation}
where the summation is over the set
\begin{equation*}
  D=\{(y_{12},y_{13},y_{23})\in\mathbb{N}^3 : \{y_{12}+y_{13}\leq x_1\} \cup \{y_{12}+y_{23}\leq x_2\} \cup\{y_{13}+y_{23}\leq x_3\} \neq\emptyset\}.
\end{equation*}
The properties of the trivariate model are essentially the same of the bivariate case~\cite{kawamura_structure_1976}. To simplify the summation over $D$, the killing barrier will be placed at $M=1$, this choice will also allow us to directly analyze the $\nth$-CDS in the next section; so the model becomes,\emph{de facto}, equivalent to the multivariate exponential variable model by Marshall and Olkin~\cite{marshall_multivariate_1967}. The probability that all $X$s are $0$ is the survival function of all three coordinates, and it is given by:
\begin{equation}
  \label{eq:S3_tri}
  S^3(t)=\ee{-at}.
\end{equation}
The  survival function of two coordinates is then given by Eq.~\eqref{eq:survival_n_general} and reads:
\begin{equation}
  \label{eq:S2_tri}
  \begin{aligned}
    S^2(t)&=S^3(t)+I_{\{(X_1=1,X_2=0,X_3=0,t)|(X_1=0,X_2=0,X_3=0,0)\}}(t)+\\
          &I_{\{(X_1=0,X_2=1,X_3=0,t)|(X_1=0,X_2=0,X_3=0,0)\}}(t)+\\
          &I_{\{(X_0=1,X_2=0,X_3=1,t)|(X_1=0,X_2=0,X_3=0,0)\}}(t)
  \end{aligned}
\end{equation}
where the three paths (see Fig.~\ref{fig:graph}) consider the cases where only one coordinate jumps to $1$. The additional addends are:
\begin{equation}
  \label{eq:G_1_1}
  \begin{aligned}
    &I_{\{(X_i=1,X_j=0,X_r=0,t)|(X_i=0,X_j=0,X_r=0,0)\}}(t)\\
    &=\int_0^t d\tau P^2_{jr}(0,0,t|0,0,\tau)P^2_{j,r}(0,0,\tau|1,\tau)F_{\lambda_i}(\tau)=\\
    &\int_0^\tau d\tau \ee{-(\lambda_j+\lambda_r+\lambda_{jr})(t-\tau)}\ee{-(\lambda_j+\lambda_r+\lambda_{jr}+\lambda_{ij}+\lambda_{ir})\tau}\lambda_i\ee{-\lambda_i \tau}\\
    &=\frac{\lambda_i}{\lambda_i+\lambda_{ir}+\lambda_{ij}}\left(\ee{-(\lambda_j+\lambda_r+\lambda_{jr})t}-\ee{-at}\right);
  \end{aligned}
\end{equation}
where the conditional probability mass function $P^2_{j,r}(0,0,\tau|1,\tau)$ can be easily obtained using the independence among all $Y$s Poisson processes and the definition of conditional probability.\\
Finally the survival function of the last coordinate is again given by Eq,~\eqref{eq:survival_n_general}:
\begin{equation}
  \label{eq:S3_tri}
  \begin{aligned}
    S^1(t)&=S^2(t)+\\
          &\sum_{C_{\{i,j,r\}\in\{1,2,3\}}} I_{\{(X_i=1,X_j=1,X_r=0,t),(X_i=1,X_j=0,X_r=0,t_1),(X_i=0,X_j=0,X_r=0,0)\}}(t)+ \\
          &I_{\{(X_i=1,X_j=1,X_r=0,t),(X_i=0,X_j=1,X_r=0,t_1),(X_i=0,X_j=0,X_r=0,0)\}}(t)+\\
    &I_{\{(X_i=1,X_j=1,X_r=0,t),(X_i=0,X_j=0,X_r=0,0)\}}(t);
  \end{aligned}
\end{equation}
where $\sum_{C_{\{i,j,r\}\in\{1,2,3\}}}$ denotes the sum over the three combinations of $\{1,2,3\}$ (in total there are nine $I$ addends, as many as the paths connecting $AAA$ to $ADD$, $DAD$, and $DDA$ in Fig.~\ref{fig:graph}). The first two $I$ addends denote the paths where two coordinates are killed at different times, while the third the path where two coordinates are killed simultaneously, this is a feature of the model. Explicitly:
\begin{equation}
  \label{eq:G_2_1}
  \begin{aligned}
    &I_{\{(X_i=1,X_j=1,X_r=0,t),(X_i=1,X_j=0,X_r=0,t_1),(X_i=0,X_j=0,X_r=0,0)\}}(t)\\
    &\int_0^t d\tau_2 \int_0^{\tau_2}d\tau_1 P^1_{r}(0,t|0,\tau_2)P^1_{r}(X_r=0,\tau_2|X_j=1,\tau_2;X_i=1,\tau_1)F_{\lambda_i}(\tau_1)F_{\lambda_j}(\tau_2)=\\
    &\lambda_i\lambda_j\ee{-\lambda_r t}\int_0^t d\tau_2 \ee{-(\lambda_j+\lambda_{jr})\tau_2}\int_0^{\tau_2}d\tau_1\ee{-(\lambda_i+\lambda_{ij}+\lambda_{ir})\tau_1}=\\
    &\frac{\lambda_i\lambda_j}{(\lambda_i+\lambda_{ir}+\lambda_{ij})(\lambda_j+\lambda_{jr})}\left(\ee{-\lambda_r t}-\ee{-(\lambda_j+\lambda_r+\lambda_{jr})t}\right)-\\
    &\frac{\lambda_i\lambda_j}{(\lambda_i+\lambda_{ir}+\lambda_{ij})(a-\lambda_r)}\left(\ee{-\lambda_r t}-\ee{-at}\right),
  \end{aligned}
\end{equation}
where the notation of the conditional probability mass has been expanded for sake of clarity. And
\begin{equation}
  \label{eq:G_2_2}
  \begin{aligned}
    &I_{\{(X_i=1,X_j=1,X_r=0,t),(X_i=0,X_j=0,X_r=0,0)\}}(t)\\
    &=\int_0^t d\tau P^1_{r}(0,t|0,\tau)P^1_{r}(0,\tau|1,1,\tau)F_{\lambda_{ij}}(\tau)=\\
    &\int_0^t d\tau \ee{-\lambda_r(t-\tau)}\ee{-(\lambda_i+\lambda_j+\lambda_{r}+\lambda_{ir}+\lambda_{jr})\tau}\lambda_{ij}\ee{-\lambda_{ij}\tau}=\frac{\lambda_{ij}}{a-\lambda_r}\left(\ee{-\lambda_r t}-\ee{-at}\right).
  \end{aligned}
\end{equation}
In Fig.~\ref{fig:survival_poisson_3D} we show the comparison between the formulae just derived with a simulation results. In particular, in the bottom panel, we see how, in the case of time-scale separation between the driving processes, the $S^1(t)$ function presents a less trivial shape with a hump.
\begin{figure}
   \centering
   \includegraphics[width=0.95\textwidth]{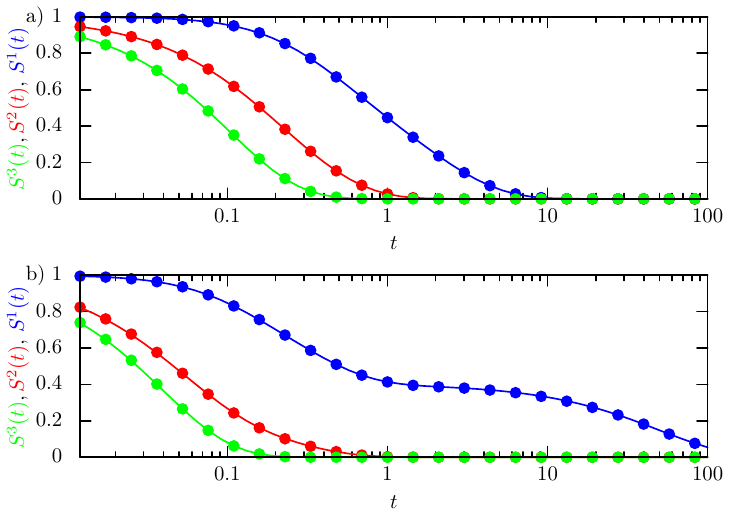}
   \caption{Comparison between a Monte Carlo simulation ($10,000$ realizations) and the theoretical survival functions of two trivariate Poisson processes. $S^3(t)$ in green, $S^2(t)$ in red, and $S^1(t)$ in blue. The dots are obtained using the simulation's results while the solid lines using our formulas. In the top panel:$\lambda_1=1.2,\lambda_2=0.5,\lambda_3=3.3,\lambda_{12}=1.4,\lambda_{13}=3.1,\lambda_{23}=0.12$, while in the bottom: $\lambda_1=5.3,\lambda_2=0.02,\lambda_3=3.3,\lambda_{12}=10.4,\lambda_{13}=5.1,\lambda_{23}=1.12$.}\label{fig:survival_poisson_3D}
 \end{figure}
 The generalization to more coordinates for the multivariate Poisson process~\cite{kawamura_structure_1979} is straightforward but tedious due to the large number of paths that must be taken into account.
 \subsection{The $\nth$-CDS pricing}
 \
 CDSs are derivative products that define two legs: the fee leg where the protection buyer pays at regular intervals $T_1,\ldots,T_L$ a given amount, the swap spread $u$, until its maturity $T=T_L$, and the protection leg where the protection seller pays a predetermined amount of money in the case of default before the maturity of the CDS reference entity. Pricing these products is of great importance to practitioners. A large amount of models and techniques have been developed for this task. Roughly speaking there are two main types of models: structural models and intensity models. Intensity models are often preferred by practitioners, and, in the following, we will concentrate on a very simple example among these. In these models the default is described by the first jump of a suitable jump process. Further information on intensity model based credit derivatives are available for example in~\cite{brigo_interest_2006}.\\
 The $\nth$-CDSs are more exotic over the counter contracts that work analogously to the single name CDSs. In the very simplified form we will use in this article\footnote{Since these are over the counter products,
   they can be customized according to the buyer and seller preferences and/or possibilities.}, these products reference $N$ different entities; the fee leg is the same with respect to the single name case, while the protection leg is triggered only when $n$ entities default. The proper definition, calibration, and use of the default correlations are the key building blocks of the problem~\cite{zhou_analysis_2001}, we will not touch the topic in this article and only focus on the valuation task. Semi-analytical formulae to price them has been developed in~\cite{hull_valuation_2004,mortensen_semi-analytical_2006} and copula models are also popular for pricing them~\cite{li_default_2000, madan_credit_2006, laurent_basket_2005}, however these authors' results consider the time-sequence of defaults (paths in $\Gamma$) differently from us. In particular Giesecke~\cite{giesecke_simple_2003} uses an intensity model that is equivalent to the multivariate Poisson model we just examined with killing barrier at $1$ for the valuation of multi-names CDSs. Note that the model we use allows for the simultaneous defaults of two entities, this is not necessarily a limitation of the model but an interesting feature with applications to engineering, insurance, and of course finance~\cite{protter_stopping_2024, gueye_dependent_2025}. Nevertheless, the model we use here is an oversimplification that may not capture all salient features of these derivatives' market and has been chosen in virtue of its analytical tameness. For example, a consequence of our setup is that after the non-simultaneous default of one component, e.g. $X_i$ jumps to $1$ because $Y_i$ in Eq.~\eqref{eq:trivariate} jumped to $1$, all the cross-Poisson processes labeled by the $i$ subscript disappear as well, they are not part anymore of the subsequent ``equation of motion''. Our way of considering the defaults' sequence naturally encodes the fact that different sequences may have different financial effects, for example is reasonable to think that: if the default of the larger or most important entity happens before the smaller entities' defaults in a multi-name CDS, then the surviving ones will feel a larger change in their intensities with respect to the case in which a less important entity defaults first. Again, we do not claim that our simple model specification is quantitatively accurate enough to describe such financial dynamics. Indeed, in our model, the intensities of the remaining components become smaller and therefore their average time to default larger. This is not necessarily the desired behavior after a default event. Nonetheless, it is  possible to modify the model introducing a default-dependent term structure of the intensities without greatly compromise the analytical tractability of the model, as long as the intensities are kept piece-wise constant.\\
 At inception, under the risk neutral measure and assuming constant short rate $r$ and independence between the interest rate process and the credit process, the fee leg can be written as~\cite{giesecke_simple_2003, brigo_interest_2006}
 \begin{equation}
   \label{eq:fee_leg}
   f=u\sum_{i=1}^L\ee{-rT_i}S^{N-n+1}(T_i),
 \end{equation}
 and the protection leg as
 \begin{equation}
   \label{eq:premium_leg}
   p=-\int_0^Tdt \ee{-r t}F^{N-n+1}(t).
 \end{equation}
 The fair value of these products is given by the sum of these two legs. However, market participants usually quote the prices of these products in terms of the swap spread, that is the value of $u$ that forces $f+p=0$ at inception Eqs.~\eqref{eq:fee_leg} and~\eqref{eq:premium_leg} can be easily computed using the results in the previous section.
 \begin{table}[h!]
   \centering
   \begin{tabular}{lccc}
     \toprule
     Default type& Model A& Model B& Model C\\
     \midrule
     $1\textsuperscript{st}$ to default& 1822    &42.48 & 42.48\\
     $2\textsuperscript{nd}$ to default& 41.84   &38.84 & 4.61\\
     $3\textsuperscript{rd}$ to default& 2.30    &0.06  & 0.66\\
     \bottomrule
   \end{tabular}
   \caption{The swap spreads $u$ for the three specifications of a CDS with three underlyings. In model A all single intensities $\lambda_i=2.5$ and cross-intensities $\lambda_{ij}=2.5$, in model B the single intensities are reduced to $0.01$, while in model C the cross-intensities are reduced to $0.01$. We report the spreads for the first, second, and third to default CDSs. $r=0.02$, $T=5$, and payments every $0.5$ time units. All units are arbitrary.
   }
   \label{tab:swap_spreads}
 \end{table}
 In Tab.~\ref{tab:swap_spreads} we show how the swap spreads of different $\nth$-CDSs change as we change the model specification, assuming three underlyings for the contracts. In model A we assume that all intensities are the same, that is simultaneous defaults and single defaults are equally likely, then the spread difference is rather large among the three contracts. On the other hand, if we make the single name defaults much less likely than the simultaneous defaults, only the gap between the second and third-to-default CDS is large (model B), while if we do the vice versa (model C) the gap between the first and the second-to-default CDS is large.

 \section{Conclusion}
 The main results of the article are theorems~\ref{theo:two} and~\ref{theo:many}. The reader may wonder what is the point of such inconvenient answers~\cite{wilf_what_1982}, for example the simulation of the blue line in Fig.~\ref{fig:survivalsf_1} is only three times slower then the implementation of Eq.~\eqref{eq:greensf_N1_2}, and the latter is infinitely slower and more complicated than the clever solution of Locatelli et al.~\cite{locatelli_active_2015} in Eq.~\eqref{eq:locatelliformula}. However, we argue that the main contribution of our theorems is that they highlight the non-Markovian nature of the problem, in fact Eq.~\eqref{eq:partial_nth_survival} shows explicitly how the history of the process enters the solution via the ``integrals over the paths'' in $\Gamma$. This path dependence has been leveraged to tackle the valuation of multi-name credit derivatives, we hope that this way of thinking about these problems can inspire more accurate pricing method in the space. We also hope that our results may spur more discoveries on a more fundamental level. We believe that our work could be used to unveil a richer phenomenology in the persistence exponents~\cite{bray_persistence_2013,andersen_persistence_2015}  or that it could inspire some approximate solutions, especially in the large number of interacting particle systems, perhaps up to the mean-field limit~\cite{guo_mean-field_2023}, or valuation of collateralized debt obligations~\cite{hull_valuation_2004, laurent_basket_2005}; these large scale applications are likely beyond any reasonable use of our formulae.
 
 \bibliographystyle{plain}
 \bibliography{nthFPT} 
\end{document}